\documentclass{llncs}

\usepackage{makeidx}
\usepackage[latin1]{inputenc}
\usepackage[T1]{fontenc}
\usepackage{amsmath, amssymb}
\usepackage{enumerate}
\usepackage[colorlinks=true, allcolors=blue]{hyperref}
\usepackage{color}
\usepackage{xspace}
\usepackage{comment}
\usepackage{booktabs}
\usepackage{graphicx}
\usepackage{adjustbox}
\usepackage{subcaption}
\usepackage{float}

\newcommand{\weq}{\ = \ }

\newcommand{\wle}{\ \le \ }

\newcommand{\wsim}{\ \sim \ }
\newcommand{\wlesim}{\ \lesssim \ }

\newcommand{\wasymp}{\ \asymp \ }
\newcommand{\wll}{\ \ll \ }
\newcommand{\wgg}{\ \gg \ }

\definecolor{defcolor}{rgb}{0,0,1}

\newcommand{\E}{\mathbb{E}}
\newcommand{\pr}{\mathbb{P}}

\newcommand{\Var}{\operatorname{Var}}

\newcommand{\nquad}{\hspace{-1em}}

\newcommand{\floor}[1]{\left\lfloor #1 \right\rfloor}

\newcommand{\cE}{\mathcal{E}}

\newcommand{\ess}{\operatorname{ess}}

\newcommand{\Hol}{\operatorname{H\ddot{o}l}}
\newcommand{\Lip}{\operatorname{Lip}}

\newcommand{\approxd}{\stackrel{\rm d}{\approx}}
\newcommand{\approxst}{\approx_{\rm st}}

\newcommand{\abs}[1]{\mathopen{}\mathclose\bgroup\left|#1\right|\egroup}

\newcommand{\gesim}{\gtrsim}
\newcommand{\lesim}{\lesssim}

\newcommand{\cB}{\mathcal{B}}
\newcommand{\cC}{\mathcal{C}}

\newcommand{\cF}{\mathcal{F}}

\newcommand{\Gsub}{G^{(n_0)}}

\spnewtheorem{assumption}[theorem]{Assumption}{\bfseries}{\itshape}

\newcommand{\Erdos}{Erd\H{o}s\xspace}

\newcommand{\Renyi}{R{\'e}nyi\xspace}

\newcommand{\glink}{\textnormal{link}}
\newcommand{\gtwostar}{\textnormal{2-star}}
\newcommand{\gthreestar}{\textnormal{3-star}}
\newcommand{\gthreepath}{\textnormal{3-path}}
\newcommand{\gthreepan}{\textnormal{3-pan}}
\newcommand{\gfourcycle}{\textnormal{4-cycle}}
\newcommand{\gfourstar}{\textnormal{4-star}}
\newcommand{\gfourpath}{\textnormal{4-path}}
\newcommand{\gchair}{\textnormal{chair}}
\newcommand{\gtriangle}{\textnormal{triangle}}

\newcommand{\glinktwo}{\glink^2}

\newif\ifshowb\showbtrue 
\newif\ifshowr\showrtrue   
\ifshowb
  
\else
  \excludecomment{bcomm}
\fi
\ifshowr
  
\else
  \excludecomment{rcomm}
\fi

\begin{document}

\title{Parameter estimators of random intersection graphs with thinned communities}

\mainmatter
\title{Parameter estimators of sparse random intersection graphs with thinned communities}
\titlerunning{Parameter estimators of random intersection graphs with thinned communities}
\author{Joona Karjalainen, Johan S.H. van Leeuwaarden and Lasse Leskel\"a}
\authorrunning{J. Karjalainen, J.S.H. van Leeuwaarden and L. Leskel\"a}
\tocauthor{Joona Karjalainen, Johan S.H. van Leeuwaarden, Lasse Leskel\"a}
\institute{
Aalto University, Finland  \\
Eindhoven University of Technology, The Netherlands
}
\maketitle

\begin{abstract}
This paper studies a statistical network model generated by a large number of randomly sized overlapping communities, 
where any pair of nodes sharing a community is linked with probability $q$ via the community. In the special case with $q=1$ the model reduces to a random intersection graph which is known to generate high levels of transitivity also in the sparse context. The parameter $q$ adds a degree of freedom and leads to a parsimonious and analytically tractable network model with tunable density, transitivity, and degree fluctuations. We prove that the parameters of this model can be consistently estimated in the large and sparse limiting regime using moment estimators based on partially observed densities of links, 2-stars, and triangles.
\end{abstract}


\section{Introduction}

Networks often display transitivity or clustering, the tendency for nodes to be connected if they share a mutual neighbor. Random graphs can statistically model networks with clustering after adding a community structure of small relatively dense subgraphs. Triangles, or other short cycles, then occur predominantly within and not between the communities, and clustering becomes tunable through adapting the community structure. 

There are various ways to install community structure, for instance by locally adding 
small dense graphs~\cite{Ball_Britton_Sirl_2013,Coupechoux_Lelarge_2014,Stegehuis_VanDerHofstad_VanLeeuwaarden_2016_Epidemic,VanDerHostad_VanLeeuwaarden_Stegehuis_2015}. This creates nonoverlapping communities. Another way is to introduce overlapping communities through a random intersection graph (RIG) which can be defined as the 2-section of a random inhomogeneous hypergraph where hyperedges correspond to overlapping communities \cite{Karonski_Scheinerman_Singer-Cohen_1999}. RIGs have attractive analytical features, for example admitting tunable transitivity (clustering coefficient) and power-law degree distributions \cite{Deijfen_Kets_2009,Bloznelis_2013,Bloznelis_Leskela_2016_long}.
However, by construction the RIG community structure is rigid, in the sense that every community corresponds to a clique.
In this paper we relax this property and consider an extension of the RIG, a thinned RIG where nodes within the same community are linked with some probability $q\in[0,1]$ via that community, independently across all node pairs. 


The RIG and thinned RIG  are known to generate high levels of transitivity, even in sparse regimes where nodes have finite mean degrees in the large-network limit \cite{Deijfen_Kets_2009,Petti_Vempala_2017}. In \cite{Petti_Vempala_2017} it is shown that the community density $q$ can be exploited to tune both triangle and 4-cycle densities. 
In this paper we also exploit the additional freedom offered by $q$, but for controlling the density of 2-stars instead of 4-cycles. We derive scaling relations between the model parameters to create large, sparse, clustered networks, in which the number of links grows linearly in the number of nodes $n$, and the numbers of 2-stars and triangles grow quadratically in $n$. We investigate a special instance of the sparse model parameterized by a triplet $(\lambda,\mu, q)$ where $\lambda$ corresponds to the mean degree and $\mu$ to the mean number of community memberships of a node.
By analyzing limiting expressions for the link, 2-star and triangle densities, we derive moment estimators for $\lambda$, $\mu$, and $q$ based on observed frequencies of 2-stars and triangles. Taken together, the densities of links, 2-stars and triangles  prove sufficient to produce tunable sparsity (mean degree), degree fluctuations and transitivity. 

This work is part of an emerging area in network science that connects high-order local network structure such as subgraphs with statistical estimation procedures.  
The triangle is the most studied subgraph, because it not only describes transitivity, but also signals hierarchy and community structure~\cite{Ravasz_Barabasi_2003}. Other subgraphs, however, such as 2-stars, bifans, cycles, and cliques are also relevant for understanding network organization \cite{Benson_Gleich_Leskovec_2016,Tsourakakis_Pachocki_Mitzenmacher_2017}. In this paper we exploit a direct connection between the model parameters and the frequencies of links, 2-stars and triangles. 
A key technical challenge is to characterize the mean and variance of the subgraph frequencies, where the latter requires frequencies of all subgraphs that can be constructed by merging two copies of the subgraph at hand \cite{Frank_1979,Picard_Daudin_Koskas_Schbath_Robin_2008,Matias_Schbath_Birmele_Daudin_Robin_2006,Ostilli_2014}.  A byproduct of our analysis yields a rigorous proof of the graph-ergodic theorem (analogous to \cite[Theorem 3.2]{Karjalainen_Leskela_2017}) stating that the observed transitivity (a large graph average) of a large graph sample is with high probability close to the model transitivity (a probabilistic average).



 
\vspace{.3cm}
\noindent{\bf Notation.}
For a probability distribution $\pi$ on the nonnegative integers, we denote the moments by
$
 \pi_r
 \weq \sum_x x^r \pi(x)
$
and the factorial moments by
$
 (\pi)_r
 \weq \sum_x (x)_r \pi(x),
$
where $(x)_r \weq x (x-1) \cdots (x-r+1)$. For sequences $a_n$ and $b_n$, we denote $a \lesim b$ when $a_n \le c b_n$ for some $c>0$ and all $n$. $a \asymp b$ means "$a \lesim b$ and $b \lesim a$". For $a_n = (1+o(1))b_n$ we use the notation $a \sim b$, and for $a_n/b_n \rightarrow 0$ we use $a \ll b$. $X=o_\pr(1)$ is read as "$X$ converges to zero in probability".

\section{Model description}
\label{sec:mod}

We will study a statistical network model with $n$ nodes (individuals, users, vertices) and $m$ overlapping communities (attributes, blocks, groups, layers). The model is parameterized by $(n, m, \pi, q)$, where $\pi$ is a probability distribution on $\{0,\dots,n\}$ such that $\pi(x)$ corresponds to the proportion of communities of size~$x$, and $q \in [0,1]$ is the probability that two nodes are linked via a particular community.

A realization of the model corresponds to a collection of random subsets $V_k$ of $\{1,\dots,n\}$ indexed by $k=1,\dots,m$ representing the communities, and a collection of symmetric binary matrices $(C_{ij,k})_{ij}$, with $i,j = 1,\dots,n$, and $k =  1, \dots, m$. These objects are used to define an undirected random graph $G$ on node set $\{1,\dots,n\}$ with adjacency matrix
\begin{equation}
 \label{eq:Model}
 G_{ij}
 \weq \max_{k=1, \dots, m} \{B_{i,k} B_{j,k} C_{ij,k} \}, \quad i \ne j,
\end{equation}
where $B_{i,k} = 1_{V_k}(i)$ indicates whether node $i$ belongs to community $k$, and $C_{ij,k} = 1$ means that $i$ and $j$ are linked via community $k$, given that both $i$ and $j$ are members of community $k$.  We assume that $V_1, \dots, V_m$ are independent random sets with a common probability density $\pr(V_i=A)=\pi(|A|) \binom{n}{|A|}^{-1}$, and that $C_{ij,k}$ are independent $\{0,1\}$-valued random integers with mean $q$. Moreover, the arrays $(V_k)$ and $(C_{ij,k})$ are assumed independent.

The special case where $q=1$ corresponds to the so-called passive random intersection graph model \cite{Bloznelis_2013,Godehardt_Jaworski_2001}.  The special case where $\pi$ is a Dirac measure has been recently studied in \cite{Petti_Vempala_2017}.
%
%
%
The binomial community size distribution $\pi(x) = \binom{n}{x}(1-p)^{n-x}p^x$ gives another important special case of the model (referred to as Bernoulli model), 
%
which allows to smoothly interpolate between a standard \Erdos--\Renyi random graph (setting $p=1$) and a binomial random intersection graphs \cite{Frieze_Karonski_2016} (with $q=1$).



\section{Analysis of local model characteristics}

\subsection{Sparse parameter regime}

In this section we analyze how the model behaves when the number of nodes $n$ is large. We view a large network as a sequence of models with parameter quadruples $(n,m,\pi, q) = (n_\nu, m_\nu, \pi_\nu, q_\nu)$ indexed by a scale parameter $\nu = 1,2,\dots$ such that $n_\nu \to \infty$ as $\nu \to \infty$. For simplicity we omit the scale parameter from the notation.

Let 
$
 p_r
 \weq {(\pi)_r} /{(n)_r}
$
denote the probability that a particular community contains a given set of $r$ nodes. Then $mp_r$ equals the mean number of communities common to a particular set of $r$ nodes, and $\binom{n}{r} p_r = {(\pi)_r}/{r!}$ equals the expected number of $r$-sets of nodes contained in a single community. Because $mp_2 q$ equals the expected number of communities through which a given node pair is linked, it is natural to assume that $mp_2 q \ll 1$ when modeling a large and sparse network.
%
%
The following result confirms this.
%
%

\begin{proposition}
\label{the:LinkProbabilityPassive}
The probability that any particular pair of distinct nodes is linked equals
$
 \pr(\glink)
 = 1 - \left( 1 - q p_2 \right)^m.
$
Furthermore, $\pr(\glink) \ll 1$ if and only if $mp_2q \ll 1$,
in which case
\begin{equation}
\label{eq:LinkProbabilityPassive}
 \pr(\glink)
 \weq \left(1+ O(mp_2q) \right) m p_2 q.
\end{equation}
\end{proposition}

\subsection{Subgraph densities}

For an arbitrary graph $R$, the \emph{$R$-covering density} of the model is defined as the expected proportion of subgraphs\footnote{By subgraph we mean any subgraph, not just the induced ones.} of $G$ that are isomorphic to $R$. By symmetry, this quantity equals the probability that $G$ contains $R$ as a subgraph, when we assume that $V(R) \subset \{1,\dots,n\}$. Note that the $K_2$-covering density of the model is just the link density analyzed in Proposition~\ref{the:LinkProbabilityPassive}. The following result describes the covering densities of connected three-node graphs.

\begin{proposition}
\label{the:ModelDensities}
The probabilities that the model in the sparse regime $m p_2 q \ll 1$ contains as subgraph the 2-star and triangle are approximately
\begin{align}
 \label{eq:Twostar}
 \pr(\gtwostar)
 &\weq \left(1+O(mp_2q) \right) q^2 \Big( m p_3 + (m)_2 p_2^2 \Big), \\
 \label{eq:Triangle}
 \pr(\gtriangle)
 &\weq \left( 1 + O(mq p_2) \right) q^3 \Big( m p_3 + 3 (m)_2 p_2 p_3 + (m)_3 p_2^3 \Big).
\end{align}
\end{proposition}

\subsection{Model transitivity}
The transitivity (or global clustering coefficient) of a graph usually refers to the proportion of triangles among unordered node triplets which induce a connected graph. The \emph{model transitivity} of a random graph is usually defined by replacing the numerator and the denominator in the latter expression by their expected values\cite{Karjalainen_Leskela_2017}. In our case, by symmetry, the model transitivity equals
$
 \tau \weq {\pr(\gtriangle)}/{\pr(\gtwostar)},
$
and is characterized by the following result in the sparse parameter regime.

\begin{proposition}
\label{the:TransitivityPassive}
The model transitivity in the sparse regime $mp_2q \ll 1$ satisfies
\[
 \tau
 \weq \frac{p_3 q}{p_3 + (m-1) p_2^2} + o(1).
\]
\end{proposition}

\begin{remark}
In the special case with $q=1$ the above result coincides with \cite[Corollary 1]{Godehardt_Jaworski_Rybarczyk_2012} and \cite[Theorem 3.2]{Bloznelis_2013}.
%
%
\end{remark}


\subsection{Degree mean and variance}

\begin{proposition}
\label{the:PassiveDegreeMeanVariance}
The degree $D$ of any particular node of the model in the sparse regime $mp_2q \ll 1$  satisfies
\begin{align*}
 \E(D)
 \wsim m n p_2 q,
 \qquad
 \Var(D)
 \wsim m n p_2 q \left( 1 + n q \left(\frac{p_3}{p_2} - p_2 \right) \right).
\end{align*}
\end{proposition}




\section{Parameter estimation of sparse models}

Our goal is to fit the model parameters to a sparse and large graph sample of known size $n$ in a consistent way. For this we impose assumptions 
on the parameter sequence $(n_\nu, m_\nu, \pi_\nu, q_\nu)$, called the balanced sparse regime.

\begin{assumption}[Balanced sparse regime]
\label{ass:Balanced}
The ratio $m/n$, the factorial moments $(\pi)_1$, $(\pi)_2$, $(\pi)_3$, and the parameter $q$ converge to nonzero finite constants as the scale parameter tends to infinity.
\end{assumption}

Propositions \ref{the:TransitivityPassive} and \ref{the:PassiveDegreeMeanVariance} imply that in the balanced sparse regime, the mean degree $\lambda$, the degree variance $\sigma^2$, and the model transitivity $\tau$ converge to nonzero finite constants which are related to the model characteristics via the  formulas
\[
 \lambda \sim (m/n)(\pi)_2 q,
 \qquad
 \sigma^2 \sim \lambda \left( 1 + q \frac{(\pi)_3}{(\pi)_2} \right),
 \qquad
 \tau \sim \frac{(\pi)_3q}{(\pi)_3 + (m/n)(\pi)_2^2}.
\]
These are the three  model characteristics we wish to fit to real data. Single-parameter distributions $\pi$ are of special interest, as the parameter then  determines both $(\pi)_2$ and $(\pi)_3$, reducing the number of unknowns by one.

\subsection{Empirical subgraph counts}

Consider the model $G=(n,m,q,\pi)$ and assume that we have observed a subgraph $\Gsub$ induced by $n_0$ nodes. We wish to estimate one or more model parameters using the empirical subgraph counts in $\Gsub$ and the asymptotic relations developed in Section 3. Computationally efficient estimators are obtained by choosing a suitably low $n_0$.

Denote by $N_{K_2}(\Gsub)$ the number of links, by $N_{S_2}(\Gsub)$ the number of (induced or noninduced) subgraphs which are isomorphic to the 2-star, and by $N_{K_3}(\Gsub)$ the number of triangles in the observed graph $\Gsub$. These are asymptotically close to the expected subgraph counts by the following theorem.
\begin{theorem}
\label{the:ObservedDensities}
Consider the model in the balanced sparse regime {\rm(Assumption \ref{ass:Balanced})}. If $(\pi)_4 \lesim 1$ and $n_0 \gg n^{1/2}$, then the number of links in the observed graph $\Gsub$ satisfies 
\begin{equation}
 \label{eq:ObservedDensityLink}
 N_{K_2}(\Gsub)
 \weq (1+o_\pr(1) ) \E N_{K_2}(\Gsub).
\end{equation}
If also $(\pi)_6 \lesim 1$, and $n_0 \gg n^{2/3}$, then
\begin{align}
 \label{eq:ObservedDensityTwostar}
 N_{S_2}(\Gsub)
 &\weq \left(1+o_\pr(1) \right) \E N_{S_2}(\Gsub), \\
 \label{eq:ObservedDensityTriangle}
 N_{K^3}(\Gsub)
 &\weq \left(1+o_\pr(1) \right) \E N_{K_3}(\Gsub).
\end{align}
\end{theorem}





\subsection{Parameter estimation in the Bernoulli model}
The binomial community size distribution $\pi(x) = \binom{n}{x}(1-p)^{n-x}p^x$ with $p \in (0,1)$ gives $(\pi)_r = n!/(n-r)!p^r$ for all integers $r \ge 1$. We parameterize the model with three positive constants $(\lambda, \mu, q)$ (with $q$ not depending on scale) and choose 
\begin{align}
\label{eq:BernoulliModel}
m \weq \floor{\frac{\mu^2 q}{\lambda} n},  \quad \mbox{and} \quad p \weq \frac{\lambda}{\mu q} n^{-1},
\end{align}
where $\mu$ can be interpreted as the mean number of communities of a node. The following (asymptotic) relations follow from the results in Section 3:
\begin{align*}
 \lambda \wsim nqmp^2, \quad
 \sigma^2 \wsim n q m p^2 \left( 1 + n q p \right), \quad
 \tau \wsim \frac{q}{1 + mp},
\end{align*}
from which one may solve
\[
\mu \weq \frac{\lambda^2}{\sigma^2-\lambda} \quad \mbox{and} \quad q \weq \tau(1+\frac{\lambda^2}{\sigma^2-\lambda}).
\]
After substituting the asymptotic densities from Section 3 and estimating them using empirical counts we obtain (after some algebra) the estimators 
\begin{gather*}
 \hat\lambda
 \weq (n-1) \binom{n_0}{2}^{-1} N_{K_2}(\Gsub), \\
\hat\mu \weq \frac{2N_{K_2}(\Gsub)^2}{n_0 N_{S_2}(\Gsub) - 2N_{K_2}(\Gsub)^2}, \quad \hat q \weq \frac{3 n_0 N_{K_3}(\Gsub)}{n_0 N_{S_2}(\Gsub)-2N_{K_2}(\Gsub)^2}.
\end{gather*}
To summarize, we estimate the parameters $\mu$ and $q$ by counting the numbers of links, 2-stars, and triangles from an induced subgraph of $n_0$ nodes. Alternatively, this can be seen as a way of fitting the transitivity and the mean and variance of the degrees. The theoretical justification is given by the following theorem.



\begin{theorem}
$\hat \lambda$, $\hat\mu$, and $\hat q$ converge in probability to the true values $\lambda$, $\mu$, and $q$, under the Bernoulli model defined by \eqref{eq:BernoulliModel} given $n_0 \gg n^{2/3}$.
\end{theorem}




\begin{proof}
The assumptions of Theorem \ref{the:ObservedDensities} and Propositions \ref{the:LinkProbabilityPassive} and \ref{the:PassiveDegreeMeanVariance} are satisfied by \eqref{eq:BernoulliModel}, which establishes the claim for $\hat \lambda$.
Dividing and multiplying both $\hat\mu$ and $\hat q$ by $n_0^2$ yields rational expressions where the numerators and denominators converge in probability to  nonzero constants by Theorem \ref{the:ObservedDensities} and Propositions \ref{the:LinkProbabilityPassive} and \ref{the:ModelDensities}. The claim now follows from the continuous mapping theorem.
\end{proof}

\section{Numerical experiments}
\subsection{Attainable regions in the Bernoulli model}
The relations $\sigma^2 \geq \lambda$, $\tau \in (0,1)$ and $\tau \leq (1+\lambda^2/(\sigma^2-\lambda))^{-1}$ restrict the attainable combinations $(\lambda, \tau, \sigma^2)$; see Figure \ref{fig:attainparam}. 
To obtain a model with a large asymptotic transitivity coefficient, one may choose a low mean degree and a large degree variance. The flexibility gained by allowing $q\leq1$ is also illustrated in Figure \ref{fig:attainprob}. The discreteness of the attainable points $(\pr(\glink), \pr(\gtriangle))$ is obvious with $q=1$, whereas the points with $q\leq1$ fill a large part above the curve $\pr(\gtriangle)=\pr(\glink)^3$. 

\begin{figure}[H]
\begin{centering}
        \includegraphics[width=0.5\textwidth]{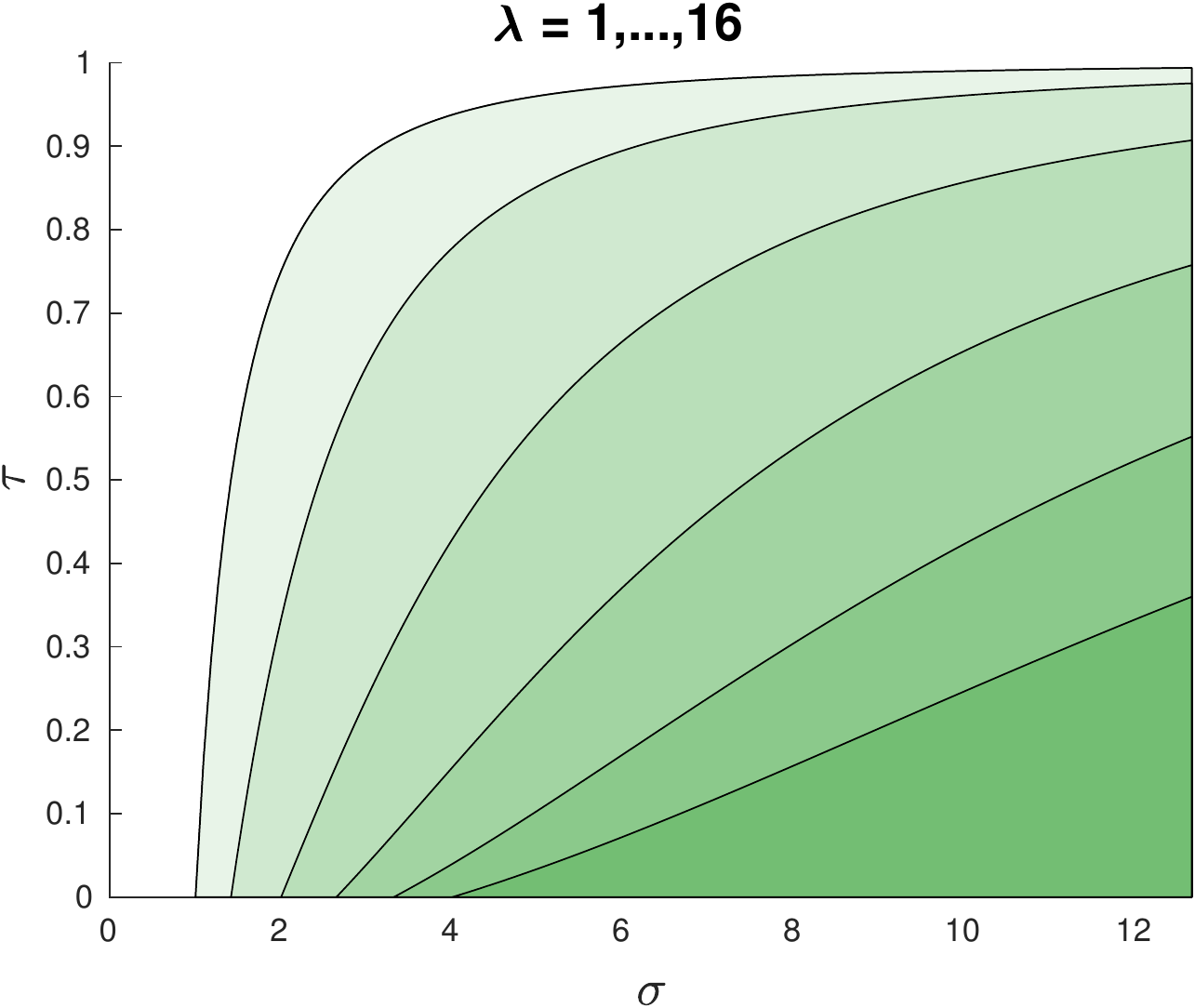}
 \caption{Attainable combinations of $(\tau, \sigma)$ for $\lambda \in \{1, 2, 4, 7, 11, 16 \}$. Combinations with $q=1$ lie on the curves. The points under the curves are obtained by setting $q\leq1$. }
 \label{fig:attainparam}
 \end{centering}
\end{figure}

\begin{figure}[H]
      \begin{subfigure}[t]{0.5\textwidth}
        \includegraphics[height=2in]{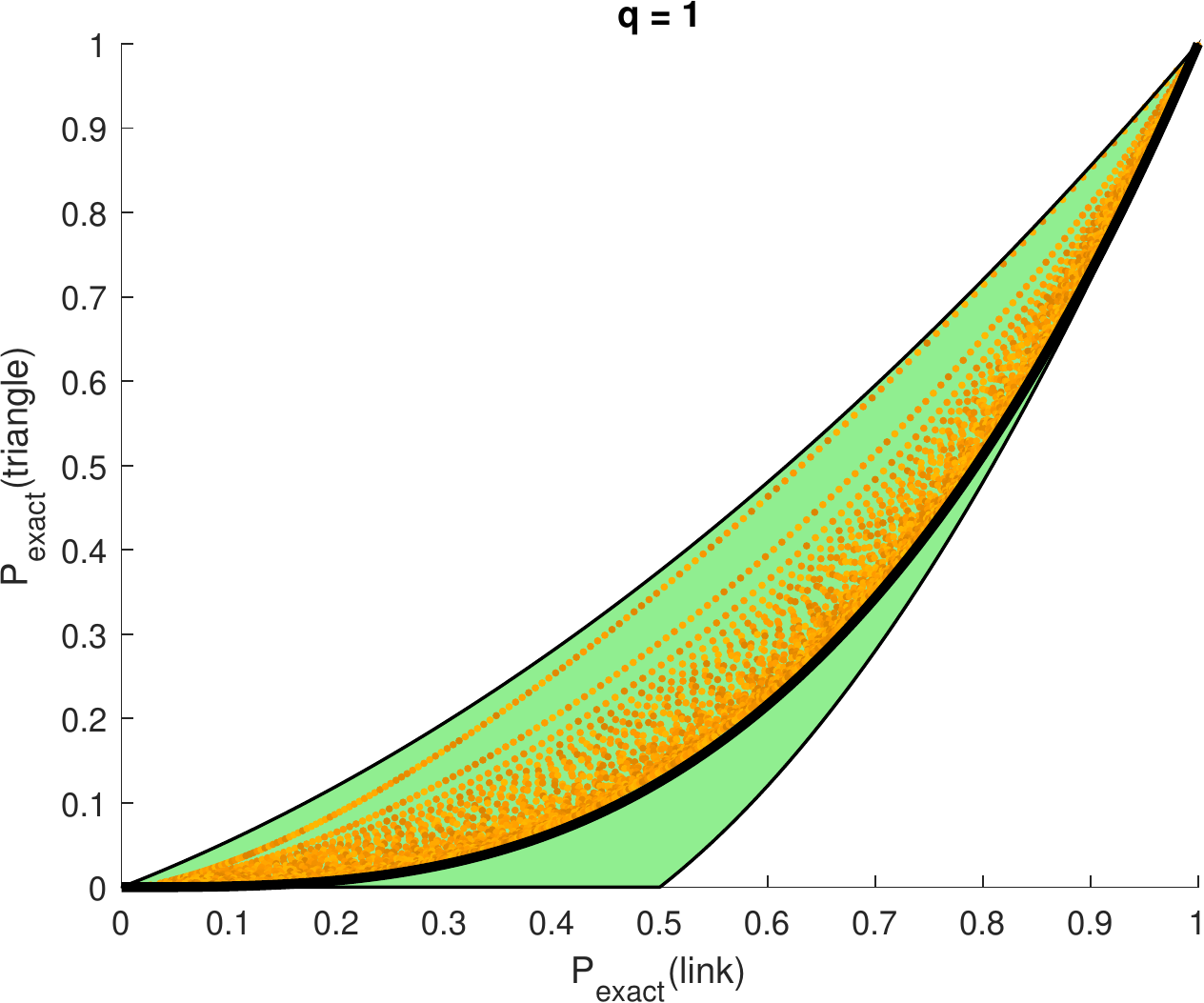}
 		\caption{}
      \end{subfigure}
      \begin{subfigure}[t]{0.5\textwidth}
        \includegraphics[height=2in]{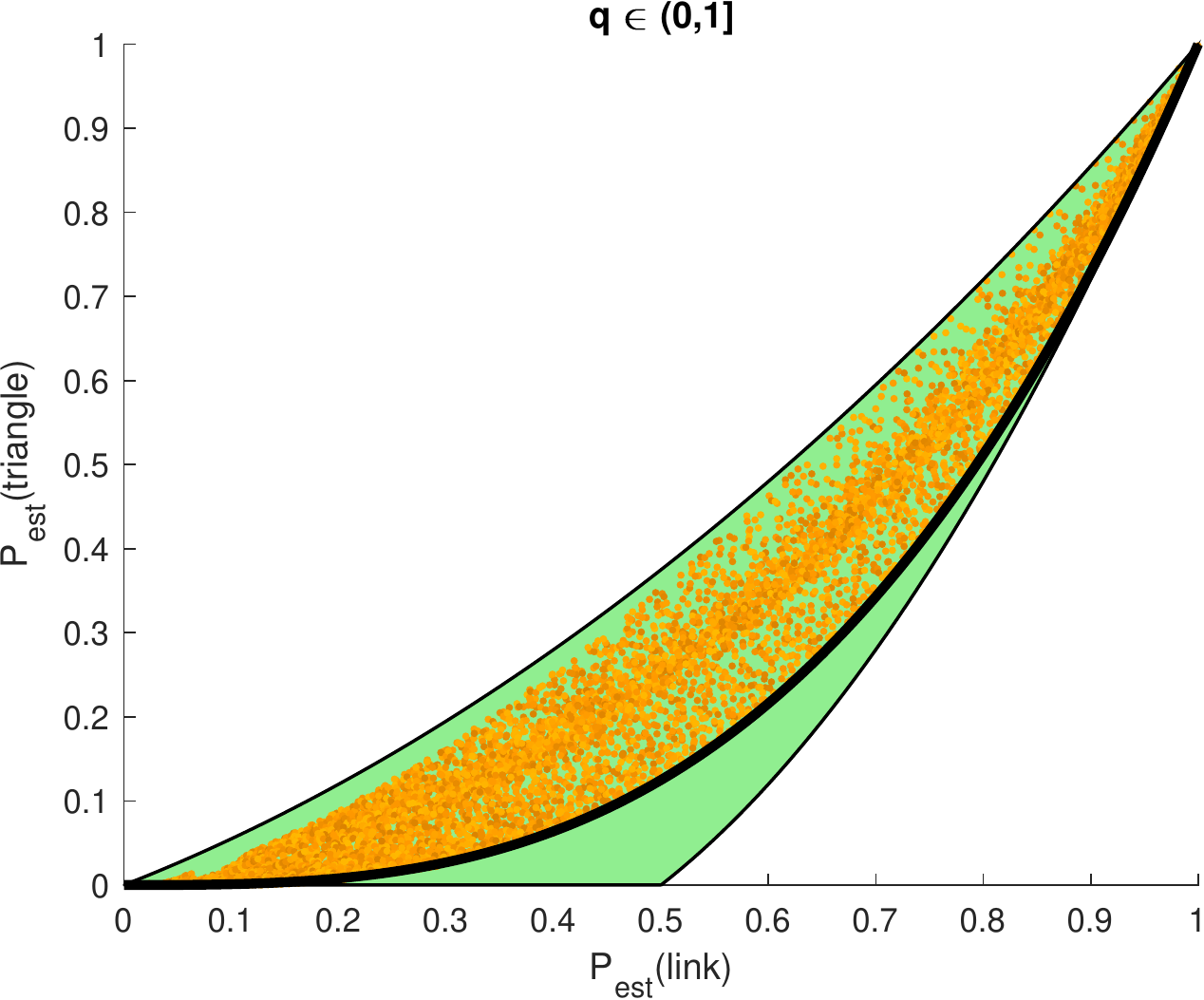}
 		\caption{}
      \end{subfigure}
\caption{Attainable combinations of link and triangle probabilities in Bernoulli models with different values of $\lambda$ and $\tau$ ($\lambda\leq500$, $\tau\geq0.0002$), and (a) $q=1$ (exact probabilities) and (b) $q\leq1$ (averages of 1000 Monte Carlo samples). The solid curves represent theoretical bounds, and the thick black curve the \Erdos--\Renyi graph. }
\label{fig:attainprob}
\end{figure}


\subsection{Real data}


Ten data sets of different sizes were analyzed using the Bernoulli model. The whole data sets were used for estimation, i.e., $n_0=n$. The obtained estimates are listed in Table \ref{tab:data}.  Because we essentially fit $\tau$ and $\lambda$, these values are listed in Table \ref{tab:data} only for illustration purposes. In the largest data sets the estimates of $q$ are very small, which might suggest that the structure of the model is not strongly supported by the data. For one of the data sets, Dolphin,  the estimate of $q$ is outside the allowed range $(0,1)$. This may be related to the denseness of the network. On the other hand, simulation results in \cite{Karjalainen_Leskela_2017} suggest that the size $n=62$ may not be sufficient for estimators based on asymptotic moment equations. 
\begin{table}[ht]
\centering
\small
{\setlength{\tabcolsep}{0.55em}
\begin{tabular}{ll|lllll|ll} 
\toprule 
Data set & $n$ & $\hat \lambda$  & $\hat \tau$ & $\hat q$ & $\hat m$ & $\hat \sigma$  &  $\hat m_{q=1}$ & $\hat \sigma_{q=1}$ \\ 
\midrule
ca-AstroPh$^1$       & 18772   & 21.1 & 0.32 & 0.47 &  100  & 30.6 & 4092 & 15.1\\
ca-HepPh$^1$         & 12008   & 19.7 & 0.66 &  0.78 & 15   & 46.6 & 162 & 28.2\\
Dolphin$^2$          & 62      & 5.1  & 0.31 & - (2.36) & 1255 & 3.0 & 61 & 4.1\\
email-Eu-core$^1$    & 1005    & 32.0 & 0.27  & 0.47 & 8 & 37.0 & 236 & 20.1\\
Facebook$^2$         & 63731   & 25.6 & 0.15 & 0.21  & 90   & 40.0 & 82756 & 11.8\\
Flickr$^1$           & 105938  & 43.7 &  0.40 & 0.46  & 23   & 115.6 & 5377 & 36.4\\
Flixster$^2$         & 2523386 & 6.3  & 0.01  & 0.014 & 4   & 36.6 & 2.1*$10^9$ & 2.6 \\
Twitter$^2$          & 2919613 & 8.8  & 0.006 & $<$0.001 &77   & 20.9 & 9.3*$10^9$ & 3.0\\
USAir97$^3$ & 332     & 12.8 &  0.40   &  0.56 & 2 & 20.1 & 60.1 & 11.0\\ 
wiki-talk$^{1}$         & 2394385 & 3.9  & 0.002 & 0.002 & $<$1    & 102.5 & 1.27*$10^{11}$ & 2.0\\ 
\bottomrule 
\end{tabular} }
\caption{Parameter estimates of the Bernoulli model for collaboration networks in astrophysics and high energy physics, a social network of bottleneck dolphins, an e-mail network from a research institution, a geographically local Facebook network, a Flickr image network, a social network of Flixster users, a Twitter network of users who mention each other in their tweets, a US aiport network, and a Wikipedia communications network. Data sets from $^1$ \cite{snapnets}, $^2$ \cite{konectdata}, and $^3$ \cite{pajekdata}.}
\label{tab:data}
\end{table}

The rightmost two columns in Table \ref{tab:data} display reference values of $m$ and $\sigma$ estimated for the RIG model ($q=1$) using the estimators introduced in \cite{Karjalainen_Leskela_2017}. These estimators give very large values for $m$ and grossly underestimate $\sigma$ in the largest data sets. These observations speak for the significantly improved model fit when using the thinned RIG model instead of the classical RIG model.

\section{Technical proofs}
\subsection{Analysis of link density}
\begin{proof}[Proof of Proposition~\ref{the:LinkProbabilityPassive}]
The probability of the event $\cE_k$ that nodes 1 and 2 are linked via community $k$ can be written as
\[
 \pr(\cE_k)
 \weq \pr\left(V_k \supset \{i,j\}, \, C_{12,k}=1 \right)
 \weq p_2 q.
\]
Because the events $\cE_1, \dots, \cE_m$ are independent, it follows that
\[
 \pr(\glink)
 \weq \pr \Big( \bigcup_k \cE_k \Big)
 \weq 1 - \prod_k \pr(\cE_k^c)
 \weq 1 - \left( 1 - p_2 q \right)^m.
\]
The inequality $1-x \le e^{-x}$ and the union bound $\pr(\cup_k \cE_k) \le \sum_k \pr(\cE_k)$ imply that
$
 1 - e^{-mp_2q}
 \le
 \pr(\glink)
 \le m p_2 q,
$
from which we see that $\pr(\glink) \ll 1$ if and only if $mp_2q \ll 1$. The approximation formula \eqref{eq:LinkProbabilityPassive} follows from the Bonferroni's bounds
\[
 m p_2 q - \binom{m}{2} (p_2 q)^2
 \wle \pr(\glink)
 \wle m p_2 q.
\]
\end{proof}

\subsection{Analysis of 2-star covering density}

\begin{proof}[Proof of Proposition~\ref{the:ModelDensities}: equation~\eqref{eq:Twostar}]
Consider a 2-star with node set $\{1,2,3\}$ and link set $\{\{1,2\}, \{1,3\}\}$. Denote by $\cB_{A,k} = \{V_k \supset A\}$ the event that community $k$ covers a node set $A$, and by $\cC_{ij,k}$ the event that $C_{ij,k} = 1$. Then $\cE_{ij,k} = \cB_{ij,k} \cap \cC_{ij,k}$ 
%
is the event that node pair $ij$ is linked by community $k$.
Then the probability that $G$ contains the 2-star as a subgraph is given by
\[
 \pr(\gtwostar)
 \weq \pr\Big(\bigcup_{k \in [m]^2} \cF_{k}\Big),
\]
where $\cF_{k} = \cE_{12,k_1} \cap \cE_{13,k_2}$ for an ordered community pair $k=(k_1,k_2)$. Observe that $\pr(\cF_{k}) = q^2 p_3$ for $k_1=k_2$ and $\pr(\cF_{k}) = q^2 p_2^2$ otherwise. Therefore,
\[
 \pr(\gtwostar)
 \wle \sum_{k \in [m]^2} \pr(\cF_{k})
 \weq m q^2 p_3 + (m)_2 q^2 p_2^2.
\]
To prove the claim using Bonferroni's bounds, it suffices to show that 
\begin{equation}
 \label{eq:TwostarUpperBound}
 \sum_{(k,\ell)} \pr( \cF_k, \cF_\ell)
 \wll q^2 \Big( mp_3 + (m)_2 p_2^2 \Big),
\end{equation}
where the sum on the left is over all $(k,\ell)$-pairs with $k,\ell \in [m]^2$ and $k \ne \ell$.

We will now compute the sum on the left side of \eqref{eq:TwostarUpperBound}. Note that
\[
 \pr( \cF_k, \cF_\ell )
 \weq q^{\abs{\{k_1,\ell_1\}} + \abs{\{k_2,\ell_2\}}}
 \pr( \cB_{12,k_1}, \cB_{13,k_2}, \cB_{12,\ell_1}, \cB_{13,\ell_2} ).
\]
Therefore, for example, for a $(k,\ell)$-pair of the form $(k_1,k_2,\ell_1,\ell_2) = (a, a, b, c)$ with distinct $a, b, c$ we have
\[
 \pr( \cF_k, \cF_\ell )
 \weq q^4 \pr( \cB_{123,a}, \cB_{12,b}, \cB_{13,c} )
 \weq q^4 p_2^2 p_3.
\]
The table below displays the values of $\pr( \cF_k, \cF_\ell )$ for all combinations of $k \ne \ell$, and the cardinalities of such combinations.

\begin{center}
\small
\begin{tabular}{lrr}
\toprule
$(k_1,k_2,\ell_1,\ell_2)$& Cardinality & $\pr( \cF_k, \cF_\ell )$ \\
\midrule
$(a, b, c, d)$ & $(m)_4$ & $q^4 p_2^4$ \\
$(a, b, a, c)$ or $(a,b,c,b)$ & $2 (m)_3$ & $q^3 p_2^3$ \\
$(a, a, b, c)$ or $(a, b, c, c)$ or $(a, b, c, a)$ or $(a, b, b, c)$ & $4 (m)_3$ & $q^4 p_2^2 p_3$ \\
$(a, a, b, b)$ or $(a, b, b, a)$ & $2(m)_2$ & $q^4 p_3^2$ \\
$(a, a, a, b)$ or $(a, a, b, a)$ or $(a,b,a,a)$ or $(b, a, a, a)$ & $4(m)_2$ & $q^3 p_2 p_3$ \\
\bottomrule
\end{tabular}
\end{center}
As a consequence,
\begin{align*}
 \sum_{(k,\ell)} \pr( \cF_k, \cF_\ell )
 \weq (m)_4 q^4 p_2^4 + 2 (m)_3 q^3 p_2^3 + 4 (m)_3 q^4 p_2^2 p_3 + 2 (m)_2 q^4 p_3^2 + 4 (m)_2 q^3 p_2 p_3
\end{align*}
By noting that $p_3 \le p_2$ and $mp_2q \ll 1$, we see that the first three terms on the right are bounded from above by $4(mp_2q)q^2 (m)_2 p_2^2$, and the last two terms on the right are bounded from above by $4 (mp_2 q) q^2 mp_3$.
Hence the above sum is at most $12 (mp_2 q) q^2 ( mp_3 + (m)_2 p_2^2)$, claim \eqref{eq:TwostarUpperBound} is valid, and the claim follows.
\end{proof}

\subsection{Analysis of triangle covering density}

\begin{proof}[Proof sketch of Proposition~\ref{the:ModelDensities}: equation~\eqref{eq:Triangle}]
Consider a triangle with node set $\{1,2,3\}$. Denote by $\cE_{e, k} = \{V_k \supset e, \, C_{e, k} = 1 \}$ the event that node pair $e$ is linked via community $k$. Then $\pr(\gtriangle) = \pr(\cup_{k \in [m]^3} \cF_k)$, where
$
 \cF_{k}
 = \cE_{12,k_1} \cap \cE_{13,k_2} \cap \cE_{23,k_3}
$
is the event that the node pairs of the triangle are linked via communities of the triplet $k=(k_1,k_2,k_3)$. Because
\[
 \pr(\cF_k)
 \weq q^3 \pr( V_{k_1} \supset 12 , V_{k_2} \supset 13, V_{k_3} \supset 23)
 \weq
 \begin{cases}
   q^3 p_3, & \quad \abs{\{k_1,k_2,k_3\}} = 1, \\
   q^3 p_2 p_3, & \quad \abs{\{k_1,k_2,k_3\}} = 2, \\
   q^3 p_2^3, & \quad \abs{\{k_1,k_2,k_3\}} = 3, \\
 \end{cases}
\]
the union bound implies that
\[
 \pr(\gtriangle)
 \wle \sum_k \pr(\cF_k)
 \wle q^3 \Big( m p_3 + 3(m)_2 p_2 p_3 + (m)_3 p_2^3 \Big).
\]
By similar techniques as in the proof of~\eqref{eq:Twostar}, one can show that
\[
 \sum_{(k,\ell): k \ne \ell} \pr(\cF_k, \cF_\ell)
 \wlesim (mq p_2) \sum_k \pr(\cF_k)
 \wll \sum_k \pr(\cF_k),
\]
and the claim follows by Bonferroni's bounds. (The details of the lengthy computations are omitted.)
\end{proof}

\subsection{Analysis of model transitivity}

\begin{proof}[Proof of Proposition \ref{the:TransitivityPassive}]
By applying Propositions~\ref{the:ModelDensities} we find that
\begin{align*}
 \tau
 \weq (1+o(1)) q \frac{m p_3 + 3 (m)_2 p_2 p_3 + (m)_3 p_2^3}{m p_3 + (m)_2 p_2^2}
 \weq (1+o(1)) q \left(   \frac{mp_3}{mp_3 + (m)_2 p_2^2} + R \right),
\end{align*}
where
\begin{align*}
 R
 \weq \frac{3 (m)_2 p_2 p_3 + (m)_3 p_2^3}{m p_3 + (m)_2 p_2^2}
 \wle m p_2 \frac{3 m p_3 + (m)_2 p_2^2}{m p_3 + (m)_2 p_2^2} 
 \wle 3 mp_2.
\end{align*}
The assumption $mq p_2 \ll 1$ now implies that $q R = o(1)$. Hence we conclude
\[
 \tau
 \weq (1+o(1)) \left( \frac{q mp_3}{mp_3 +(m)_2 p_2^2} + o(1) \right)
 \weq \frac{q p_3}{p_3 + (m-1) p_2^2}+ o(1).
\]
\end{proof}

\subsection{Analysis of degree moments}
\begin{proof}[Proof of Proposition~\ref{the:PassiveDegreeMeanVariance}]
By expressing the degree of node $i$ using the adjacency matrix as $D = \sum_{j \ne i} G_{i,j}$ and taking expectations, we find that
\begin{align*}
 \E(D) &\weq (n-1) \pr(\glink), \\
 \E(D^2) &\weq (n-1) \pr(\glink) + (n-1)(n-2) \pr(\gtwostar).
\end{align*}
By Propositions~\ref{the:LinkProbabilityPassive} and \ref{the:ModelDensities} we find that
\begin{align*}
 \pr(\glink) &\weq (1+O(mp_2q)) m p_2 q, \\
 \pr(\gtwostar) - \pr(\glink)^2
 &\weq (1+O(mp_2q)) q^2 \Big( mp_3 + (m)_2 p_2^2 - m^2 p_2^2 \Big).
\end{align*}
Hence $\E(D) \sim mnp_2 q$, and by the formula $\Var(D) = \E(D^2) - (\E D)^2$,
\begin{align*}
 \Var(D)
 &\weq (1+O(n^{-1}) \Big( n \pr(\glink) + n^2 \Big( \pr(\gtwostar) - \pr(\glink)^2 \Big) \Big) \\
 &\weq (1+O(n^{-1}) (1+O(mp_2q)) \Big( mn q p_2 + m n^2 q^2 ( p_3 - p_2^2 ) \Big).
\end{align*}
\end{proof}

\subsection{Analysis of observed link density}

\begin{proof}[Proof of Theorem~\ref{the:ObservedDensities}: equation~\eqref{eq:ObservedDensityLink}]
Let us denote by $\hat N = N_{K_2}(\Gsub)$ the number of links in the observed graph $\Gsub$. The assumptions $(\pi)_2 \gesim 1$ and $(\pi)_4 \lesim 1$ imply that $p_2 \asymp n^{-2}$ and $p_r \lesim n^{-r}$ for $r=3,4$.  Because $m \asymp n$, and $q \gesim 1$, with the help of Proposition~\ref{the:LinkProbabilityPassive}, we see that
\[
 \pr(\glink)
 \weq (1+o(1))m p_2 q
 \wasymp n^{-1},
\]
and
\[
 \E \hat N
 \weq \binom{n_0}{2} \pr(\glink)
 \wasymp n_0^2 n^{-1}
 \wgg 1.
\]

Denote by $\pr(\glinktwo)$ the probability that $G$ contains any particular pair of disjoint node pairs (e.g., pairs \{1,2\} and \{3,4\}). Note that
\begin{align*}
 \Var(\hat N)
 &\weq \sum_e \sum_{e'} \pr( e \in E(\Gsub), e' \in E(\Gsub) ) - \binom{n_0}{2}^2 \pr(\glink)^2 \\
 &\weq \binom{n_0}{2} \pr(\glink) + (n_0)_3 \pr(\gtwostar) + \binom{n_0}{2}\binom{n_0-2}{2} \pr(\glinktwo) - \binom{n_0}{2}^2 \pr(\glink)^2 \\
 &\wle n_0^2 \pr(\glink) + n_0^3 \pr(\gtwostar) + \binom{n_0}{2}^2 \Big( \pr(\glinktwo) - \pr(\glink)^2 \Big).
\end{align*}
Note that $\pr(\glink) \asymp n^{-1}$ and $\pr(\gtwostar) \lesim n^{-2}$.
Furthermore,
\begin{align*}
 \pr(\glinktwo)
 &\weq \pr( \cup_k \cup_\ell \{V_k \supset \{1,2\}, C_{12,k}=1, V_\ell \supset \{3,4\}, C_{34,\ell}=1\}) \\
 &\wle \sum_k \sum_\ell \pr(\{V_k \supset \{1,2\}, C_{12,k}=1, V_\ell \supset \{3,4\}, C_{34,\ell}=1\}) \\
 &\weq (m)_2 p_2^2 q^2 + m p_4 q^2 \weq (1+o(1)) \pr(\glink)^2 + O(n^{-3}),
\end{align*}
so that
\begin{align*}
 \Var(\hat N)
 &\wlesim n_0^2 n^{-1} + n_0^3 n^{-2} + n_0^4 n^{-3} + o(1) (\E \hat N)^2 \\
 &\wle 3 n_0^2 n^{-1} + o(1) (\E \hat N)^2 
 \wll (\E \hat N)^2.
\end{align*}

\end{proof}

\subsection{Analysis of observed 2-star covering density}

\begin{proof}[Proof sketch of Theorem~\ref{the:ObservedDensities}: equation~\eqref{eq:ObservedDensityTwostar}]
Let us denote $\hat N = N_{S_2}(\Gsub)$. Note that
\[
 \hat N
 \weq \sum_R 1_{A_R}
\]
where the sum ranges over the set of all $S_2$-isomorphic subgraphs of $K_{[n_0]}$, and $1_{A_R}$ is the indicator of the event $A_R$ that $\Gsub$ contains $R$ as a subgraph. The assumptions $(\pi)_2 \gesim 1$ and $(\pi)_6 \lesim 1$ imply that $p_2 \asymp n^{-2}$ and $p_r \lesim n^{-r}$ for $r=3,\dots,6$.  Because $m \asymp n$, and $q \gesim 1$, with the help of Proposition~\ref{the:ModelDensities}, we see that
\begin{equation}
 \label{eq:TwostarAsymp}
 \pr(\gtwostar)
 \weq q^2 \Big( m p_3 + (m)_2 p_2^2 \Big)(1+o(1))
 \wasymp n^{-2},
\end{equation}
and
\[
 \E \hat N
 \weq 3 \binom{n_0}{3} \pr(\gtwostar)
 \wasymp n_0^3 n^{-2}
 \wgg 1.
\]
The above relation underlines the role of assumption $n_0 \gg n^{2/3}$. This guarantees that there are lots of (dependent) samples to sum in the observed graph.

Let us next analyze the variance of $\hat N$.
By applying the formula $\Var(\hat N) = \E(\hat N^2) - (\E \hat N)^2$ and noting that $A_R \cap A_{R'} = A_{R \cup R'}$, we see that
\begin{align*}
 \Var(\hat N)
 &\weq \sum_R \sum_{R'} \pr(A_R, A_{R'}) - \sum_R \sum_{R'} \pr(A_R) \pr(A_{R'})
 \weq \sum_{i=0}^3 M_i,
\end{align*}
where
\begin{equation}
 \label{eq:UpperBoundMi}
 M_i
 \weq \sum_R \sum_{R': |V(R) \cap V(R')| = i} \nquad \Big( \pr(A_{R \cup R'}) - \pr(A_R)^2 \Big).
\end{equation}
For $i \ge 1$, we approximate $M_i$ from above by omitting the $\pr(A_R)$ term in \eqref{eq:UpperBoundMi}. 
By generalizing the analytical technique used in \cite{Karjalainen_Leskela_2017} (details will be available in the extended version), it can be shown that for any graph $R$ such that $|V(R)| \le 6$,
\begin{equation}
 \label{eq:KappaBoundSketch}
 \pr(A_R)
 \wlesim n^{-\kappa(R)},
\end{equation}
where $\kappa(R) = \min_{\cE}(||\cE|| - |\cE|)$, with the minimum taken across all partitions of $E(R)$ into nonempty sets, where $|\cE|$ is the number of parts in the partition, and we set $||\cE|| = \sum_{E \in \cE} |E^\flat|$ where $E^\flat = \cup_{e \in E} e$ denotes the set of nodes covered by the node pairs of $E$, so that for example, $\{\{1,2\}\}^\flat = \{1,2\}$ and $\{ \{1,3\}, \{2, 3\} \}^\flat = \{1,2,3\}$. Table~\ref{tab:2-stars} summarizes the values of $\kappa(R)$ for the type of graphs that can be obtained as unions of two 2-stars. By applying~\eqref{eq:KappaBoundSketch}, it follows that
\begin{align*}
 M_1 &\wlesim n_0^5 \Big( \pr(\gfourstar) +  \pr(\gfourpath) + \pr(\gchair) \Big) \wlesim n_0^5 n^{-4}, \\
 M_2 &\wlesim n_0^4 \Big( \pr(\gthreestar) + \pr(\gthreepath) + \pr(\gthreepan) + \pr(\gfourcycle)\Big) \wlesim n_0^4 n^{-3}, \\
 M_3 &\wlesim n_0^3 \Big( \pr(\gtwostar) +  \pr(\gtriangle) \Big) \wlesim n_0^3 n^{-2}.
\end{align*}
Because $n_0 \gg n^{2/3}$, it follows that $M_i \lesim n_0^3 n^{-2} \ll (\E \hat N)^2$ for $i=1,2,3$.

\begin{table}[t]
\begin{center}
\scriptsize
\begin{minipage}{50mm}
\includegraphics[width=\textwidth]{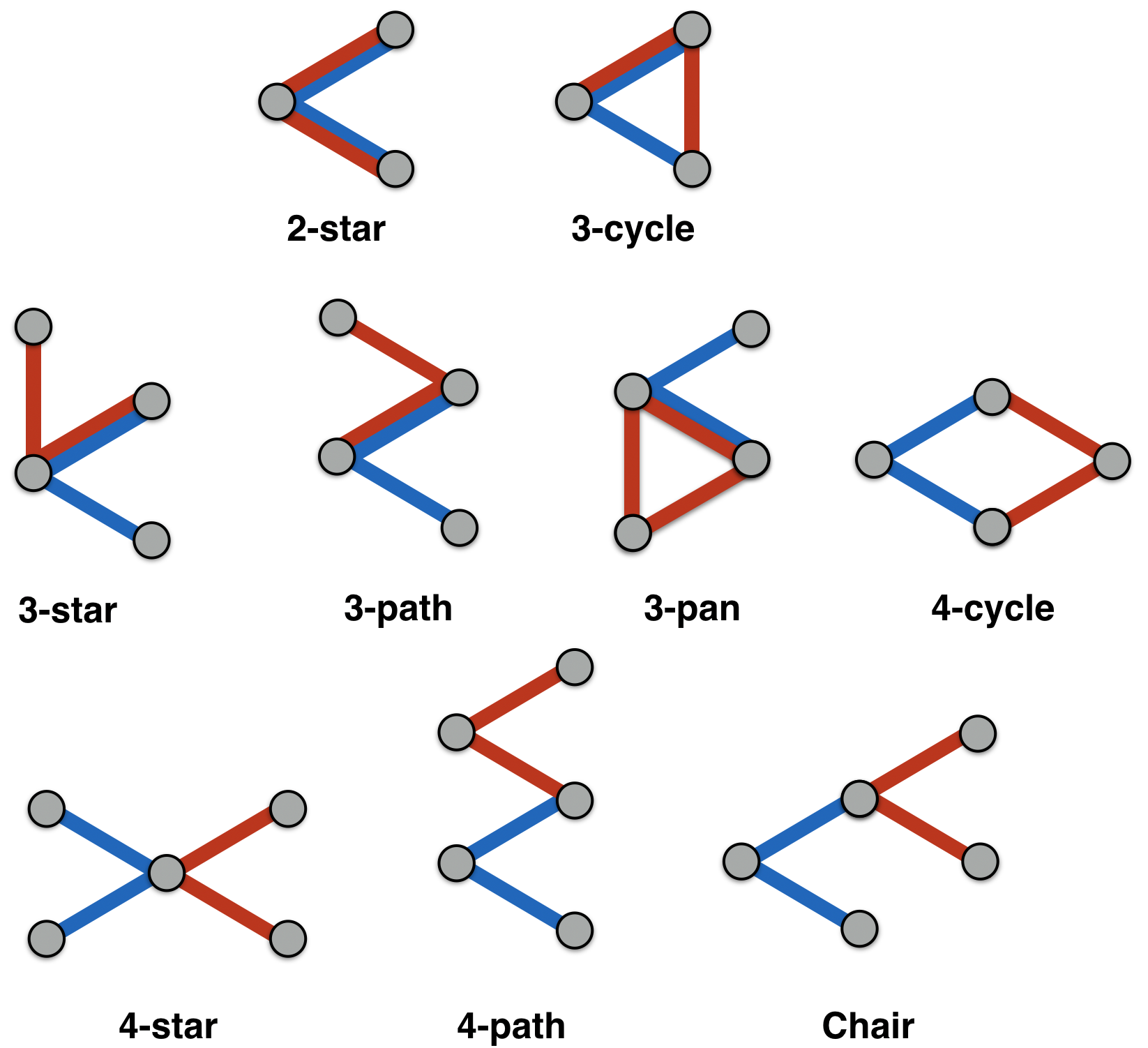}
\end{minipage}
\begin{tabular}{lccc}
\toprule
$R$ & $|V(R)|$ & $|E(R)|$ & $\kappa(R)$ \\
\midrule
2-star & 3 & 2 & 2  \\ 
3-cycle & 3 & 3 & 2  \\ 
3-star & 4 & 3 & 3  \\ 
3-path & 4 & 3 & 3  \\ 
3-pan & 4 & 4 & 3  \\ 
4-cycle & 4 & 4 & 3  \\ 
4-star & 5 & 4 & 4  \\ 
4-path & 5 & 4 & 4  \\ 
Chair & 5 & 4 & 4  \\ 
Disjoint 2-stars & 6 & 4 & 4  \\ 
\bottomrule
\end{tabular}
\caption{\label{tab:2-stars} Values of $\kappa(R)$ (obtained using an exhaustive computer search) for graphs obtained as unions of two 2-stars.}
\end{center}
\end{table}

The $M_0$-term in the variance formula \eqref{eq:UpperBoundMi} satisfies
\[
 M_0
 \wasymp n_0^6 \Big( \pr(\gtwostar^2) - \pr(\gtwostar)^2 \Big)
\]
where $\pr(\gtwostar^2)$ indicates the probability that $G$ contains a particular union of two disjoint 2-stars as a subgraph. Here we need more careful analysis because the technique used to bound $M_i$ for $i \ge 1$ would only yield an upper bound for $M_0$ of the same order as $(\E \hat N)^2$.
%
Nevertheless, a tedious but straightforward computation (details will be available in the extended version) involving all 15 partitions of the link set of a union of two disjoint 2-stars can be used to verify that
\begin{align*}
 \pr(\gtwostar^2)
 &\wle q^4 \Big( m^2 p_3^2 + 2 m^3 p_2^2 p_3 + m^4 p_2^4 \Big) + O(n^{-5}) \\
 &\weq (1+o(1)) \pr(\gtwostar)^2 + O(n^{-5}).
\end{align*}
By comparing this with \eqref{eq:TwostarAsymp}, we find that
$
 \pr(\gtwostar^2) - \pr(\gtwostar)^2
 \ll \pr(\gtwostar)^2,
$
and
\[
 M_0
 \wasymp n_0^6 \Big( \pr(\gtwostar^2) - \pr(\gtwostar)^2 \Big)
 \wll n_0^6 \pr(\gtwostar)^2
 \wasymp (\E \hat N)^2.
\]
We may now conclude that
$
 \Var(\hat N)
 = \sum_{i=0}^3 M_i
 \ll (\E \hat N)^2,
$
and hence the claim follows by Chebyshev's inequality.
\end{proof}

\subsection{Analysis of observed triangle density}

\begin{proof}[Proof sketch of Theorem~\ref{the:ObservedDensities}: equation~\eqref{eq:ObservedDensityTriangle}]
Let us denote by $\hat N = N_{K_3}(\Gsub)$ the number of triangles in the observed graph $\Gsub$.
The assumptions $(\pi)_2 \gesim 1$ and $(\pi)_6 \lesim 1$ imply that $p_2 \asymp n^{-2}$ and $p_r \lesim n^{-r}$ for $r=3,\dots,6$.  Because $m \asymp n$, and $q \gesim 1$, with the help of Proposition~\ref{the:ModelDensities}, we see that
\[
 \pr(\gtriangle)
 \weq (1+o(1))m p_3 q^3
 \wasymp n^{-2},
\]
and
\[
 \E \hat N
 \weq \binom{n_0}{3} \pr(\gtriangle)
 \wasymp n_0^3 n^{-2}
 \wgg 1.
\]

To show that $\hat N$ is with high probability close to $\E \hat N$, by Chebyshev's inequality it suffices to verify that  $\Var(\hat N) \ll (\E \hat N)^2$.  By applying the formula $\Var(\hat N) = \E \hat N^2 - (\E \hat N)^2$ and noting that $A_R \cap A_{R'} = A_{R \cup R'}$, we see that
\begin{align*}
 \Var(\hat N)
 &\weq \sum_R \sum_{R'} \pr(A_R, A_{R'}) - \sum_R \sum_{R'} \pr(A_R) \pr(A_{R'})
 \weq \sum_{i=0}^3 M_i,
\end{align*}
where
\[
 M_i
 \weq \sum_R \sum_{R': |V(R) \cap V(R')| = i} \Big( \pr(A_{R \cup R'}) - \pr(A_R)^2 \Big).
\]
In analogy with the proof of \eqref{eq:ObservedDensityTwostar} one can show (details omitted) that $M_i \ll (\E \hat N)^2$ for $i=1,2,3$ by analyzing the subgraph containment probabilities of $G$ for unions of two triangles. Again, the $M_0$ term requires special attention. A careful analysis of the various patterns through which the communities of the model can cover the links of two disjoint triangles (details available in the extended version) shows that
\[
 \pr(\gtriangle^2) - \pr(\gtriangle)^2
 \wll \pr(\gtriangle)^2.
\]
This implies $M_0 \ll (\E \hat N)^2$ and allows to conclude that 
$
 \Var(\hat N)
 \weq \sum_{i=0}^3 M_i
 \wll (\E \hat N)^2.
$
Hence the claim follows by Chebyshev's inequality.

\end{proof}

\bibliographystyle{splncs}
\bibliography{lslReferences-Joona-additions}

\end{document}